\documentclass[letterpaper, 10 pt, conference]{ieeeconf} 

\IEEEoverridecommandlockouts                              % This command is only needed if 
\usepackage{amsmath}
\usepackage{amssymb}
\usepackage[svgnames]{xcolor}
\usepackage{amsfonts}
\usepackage{footnote}
\usepackage{indentfirst}
\usepackage{enumerate}
\usepackage{bbm}
\usepackage{graphicx}
\usepackage{array} 
\usepackage{adjustbox}
\usepackage{color}
\usepackage{array,colortbl}
\usepackage{multicol,gensymb,multirow}
\usepackage{threeparttable,sidecap,xspace}
\usepackage{color, colortbl}

\usepackage{version}
\includeversion{arxiv}

\definecolor{Gray}{gray}{0.9}

\newcommand{\e}{\mathop{e}}
\newcommand{\Id}{\mathop{I\mkern-2.5mu{d}}}
\newcommand{\bigO}{\mathcal{O}}

\usepackage[font=small]{caption}
\usepackage{subfigure}
\newcommand{\suchthat}{\;\ifnum\currentgrouptype=16 \middle\fi|\;}
\newcommand{\until}[1]{\{1,\dots, #1\}}
\newcommand{\subscr}[2]{#1_{\textup{#2}}}

\newcommand{\setdef}[2]{\{#1 \; | \; #2\}}

\newcommand{\map}[3]{#1: #2 \rightarrow #3}

\newcommand{\real}{\mathbb{R}}

\newcommand{\scirc}{\raise1pt\hbox{$\,\scriptstyle\circ\,$}}

\newcommand\oprocendsymbol{\hbox{$\square$}}
\newcommand\oprocend{\relax\ifmmode\else\unskip\hfill\fi\oprocendsymbol}

\newcommand{\myclearpage}{\clearpage}
\renewcommand{\myclearpage}{}

\newcommand{\Iinfty}{I_{\infty}}

\newcommand{\sign}{\operatorname{sign}}

\renewcommand{\baselinestretch}{0.98}

\DeclareSymbolFont{bbold}{U}{bbold}{m}{n}
\DeclareSymbolFontAlphabet{\mathbbold}{bbold}
\newcommand{\vect}[1]{\mathbbold{#1}}

\newcommand{\vectorzeros}[1][]{\vect{0}_{#1}}

\newcommand{\ds}{\displaystyle}

\newcommand{\myemph}[1]{\emph{#1}}

\newtheorem{theorem}{Theorem}

\newtheorem{example}[theorem]{Example}

\newtheorem{remark}[theorem]{Remark}

\graphicspath{{./images/}{./fig/}}

\newcommand{\jac}[1]{D\mkern-2.5mu{#1}}
\newcommand{\jacfrac}[2]{\frac{\partial{#1}}{\partial{#2}}}

\newcommand{\WP}[2]{\left\llbracket{#1}, {#2}\right\rrbracket}

\newcommand{\GB}{Gr\"onwall\xspace}
\newcommand{\seminorm}[1]{{\left\vert\kern-0.25ex\left\vert\kern-0.25ex\left\vert #1 
		\right\vert\kern-0.25ex\right\vert\kern-0.25ex\right\vert}}

\newcommand{\semimeasure}[1]{\mu_{\seminorm{\cdot}}\kern-0.5ex\left(#1\right)}

\newcommand{\inprod}[2]{\langle\!\langle{#1},{#2}\rangle\!\rangle}

\usepackage{hyperref}%
\hypersetup{colorlinks=true, linkcolor=blue, breaklinks=true, urlcolor=blue}
\usepackage{stmaryrd}

\newcommand{\norm}[2]{\|#1\|_{#2}}

\newcommand{\zero}{\operatorname{zero}}
\newcommand{\fixed}{\operatorname{fixed}}

\usepackage{hyperref}%
\hypersetup{colorlinks=true, linkcolor=blue, breaklinks=true, urlcolor=blue, citecolor=blue}

\DeclareSymbolFont{bbold}{U}{bbold}{m}{n}
\DeclareSymbolFontAlphabet{\mathbbold}{bbold}

\renewcommand{\theenumi}{(\roman{enumi})}

\newcommand*{\mydoi}[1]{\href{http://dx.doi.org/#1}{\includegraphics[width=.75em]{doi.png}}}
 \usepackage{doi}

\title{From Contraction Theory to Fixed Point Algorithms\\ on Riemannian and
  Non-Euclidean Spaces}

\author{Francesco Bullo, Pedro Cisneros-Velarde, Alexander Davydov, Saber
  Jafarpour\thanks{Center of Control, Dynamical Systems and Computation,
    University of California, Santa Barbara, 93106-5070, USA. This work was
    supported in part by the Defense Threat Reduction Agency under Contract
    No.~HDTRA1-19-1-0017. Figures 1 and 2 are licensed under the CC BY-SA
    4.0. {({\tt bullo@ucsb.edu, pacisne@gmail.com, \{davydov,
        saber\}@ucsb.edu})} }}

\begin{document}
\maketitle

\begin{abstract}
  The design of fixed point algorithms is at the heart of monotone operator
  theory, convex analysis, and of many modern optimization problems arising
  in machine learning and control.  This tutorial reviews recent advances
  in understanding the relationship between Demidovich conditions,
  one-sided Lipschitz conditions, and contractivity theorems. We review the
  standard contraction theory on Euclidean spaces as well as little-known
  results for Riemannian manifolds. Special emphasis is placed on the
  setting of non-Euclidean norms and the recently introduced weak pairings
  for the $\ell_1$ and $\ell_\infty$ norms. We highlight recent results on
  explicit and implicit fixed point schemes for non-Euclidean contracting
  systems.
\end{abstract}

\thispagestyle{empty}
\pagestyle{empty}
\section{Introduction}
Motivated by control, optimization, and machine learning applications, this
document provides a simplified and incomplete tutorial about the main
contraction theorem and resulting fixed point algorithms.  The combination
of contraction theory and fixed point algorithms originates in the classic
ground-breaking paper by Desoer and Haneda~\cite{CAD-HH:72}; these ideas
play a central role in numerical integration of differential
equations~\cite{EH-GW:96}.

%% Several aspects of the interplay between monotone operator theory and
%% convex optimization are presented. The crucial role played by monotone
%% operators in the analysis and the numerical solution of convex minimization
%% problems is emphasized.

The importance of fixed point strategies in modern day data science is
described in the recent review~\cite{PLC-JCP:21}. \cite{EKR-SB:16} is a
recent survey on monotone operators and their application to convex
optimization. In this paper, we argue that contraction theory for vector
fields is the continuous-time equivalent of these theories.  Indeed,
strongly monotone operators and gradient vector fields of strongly convex
functions are strongly contracting vector fields, modulo a sign change.  A
central problem in these fields is the design of efficient fixed point
algorithms; recent contributions in this spirit
are~\cite{PMW-JJES:20,MR-RW-IRM:20}.

Of special interest in this paper are contracting systems in non-Euclidean
spaces, i.e., vector fields whose flow is a contraction mapping with
respect to a non-Euclidean norm.  In this context, Aminzare and Sontag were
the first to highlight the connections between contraction theory and
semi-inner products in~\cite{ZA-ES:13,ZA-EDS:14b}.

This tutorial is based upon the theory of weak pairings recently developed
in~\cite{AD-SJ-FB:20o,SJ-AD-FB:20r,SJ-AD-AVP-FB:21f}.  We remark
that monotone operators over smooth semi-inner product spaces are studied
for example in \cite{NKS-RNM-CN-SN:14}; here we are precisely interested in
nonsmooth polyhedral norms, such as the $\ell_1$ and $\ell_\infty$
norms. For the same reason (lack of differentiability), contraction theory
over Finsler manifolds does not directly apply to the non-Euclidean
problems of interest here.

This document also briefly reviews some generalizations to Riemannian
manifolds.  Contraction theory on Riemannian manifolds originates in the
influential work by Lohmiller and Slotine~\cite{WL-JJES:98}. A formal
coordinate-free analysis (with connection to monotone operators) is given
in \cite{JWSP-FB:12za}.  In the differential geometry literature, the study
of geodesically monotonic vector fields initiated in \cite{SZN:99b} and
relevant extensions were obtained in
\cite{JXDCN-OPF-LRLP:02,JHW-GL-VMM-CL:10}.

This document is intended to be a tutorial and makes the following
contributions. First, we provide a unified view of the main theorem on
contraction and incremental stability in the context of Euclidean,
Riemannian and non-Euclidean spaces. Similarly, we present a unified
investigation into fixed point algorithms over these three domains. Second,
we consider the setting of strongly contracting vector fields with respect
to non-Euclidean norms: we analyze and establish convergence factors for
the explicit Euler (from~\cite{SJ-AD-AVP-FB:21f}), explicit
extragradient, and implicit Euler algorithms.  Notably, these results
provide a starting point for the generalization of convex analysis and
monotone operator theory to the setting of strongly contracting vector
fields with respect to the norms $\ell_1$ and $\ell_\infty$.  Finally, we
include a number of conjectures that will hopefully stimulate further
research.
 
%% All missing proofs will be provided in a forthcoming technical report.

\myclearpage
\subsection*{A brief review of matrix measures}
We recall the standard $\ell_p$ induced norms, $p\in\{1,2,\infty\}$:
%% Given a norm $\|\cdot\|$ on $\real^n$ and its induced norm on
%% $\real^{n\times{n}}$. We recall
\begin{gather*}
  \norm{A}{2} = \sqrt{ \subscr{\lambda}{max}(A^\top A)  }, \\
    \norm{A}{1} = \max_{j\in\until{n}} \sum_{i=1}^n |a_{ij}|,  \quad
  \norm{A}{\infty} = \max_{i\in\until{n}} \sum_{j=1}^n |a_{ij}|.
\end{gather*}
where $\subscr{\lambda}{max}(A^\top A)$ is the largest eigenvalue of
$A^\top A$.  The \myemph{matrix measure} of $A \in \real^{n \times n}$ with
respect to a norm $\norm{\cdot}{}$ is
\begin{equation}
  \mu(A) := \lim_{h \to 0^+} \frac{\|I_n + hA\| - 1}{h}.
\end{equation}
From~\cite{CAD-HH:72} we recall $\mu_2(A)=\tfrac{1}{2}\subscr{\lambda}{max}\big(A+A^\top\big)$,
\begin{align*}
  \mu_1(A) \! &= \!\! \max_{j\in\until{n}}
  \Big( a_{jj} + \!\!\! \sum_{i=1,i\neq j}^n \!|a_{ij}| \Big),\enspace
  \mu_\infty(A)=\mu_1(A^\top).
\end{align*}
%%  \mu_{\infty}(A) &= \max_{i\in\until{n}} \Big( a_{ii} + \sum\nolimits_{j=1,j\neq i}^n |a_{ij}| \Big).
For $R$ invertible square, we define $\norm{A}{p,R}=\norm{RA}{p}$ and its
associated matrix measure $\mu_{p,R}(A) =\mu_p(RAR^{-1})$.  For
$P=P^\top\succ0$, we write
$\norm{x}{P}^2=\norm{x}{2,P^{1/2}}^2=x^\top{P}x$.  Matrix measures enjoy
numerous properties~\cite{CAD-HH:72}; we present here only the so-called
Lumer's equalities:
\begin{subequations} \label{eq:Lumer-equalities}
  \begin{align}
  \mu_{2,P^{1/2}}(A) &= \max_{\norm{x}{2,P^{1/2}} = 1} x^\top PA x   \label{eq:Lumer-lemma} \\
   &= \min\setdef{b\in\real}{A^\top P + PA\preceq 2b P}.
  \end{align}
\end{subequations}

\myclearpage
\section{Contraction and Monotone Operators on the Euclidean Space $(\real^n,\ell_2)$}

We start with a very simple motivating discussion.  For $b\in\real$,
$\map{f}{\real}{\real}$ is \myemph{one-sided Lipschitz (osL)} if
\begin{align} 
  & (x-y)(f(x)-f(y)) \leq b (x-y)^2, \quad &&\forall x,y \\
  &  \iff \quad f(x)-f(y) \leq b (x-y) , \quad &&\forall x>y %\\
%  &  \iff \quad      f'(x) \leq b,  && \forall x \label{eq:dosL1}
\end{align}
and if $f$ is continuously differentiable
\begin{align} 
  &  \iff \quad      f'(x) \leq b,  && \forall x \label{eq:dosL}
\end{align}
We refer to~\eqref{eq:dosL} as \myemph{differential one-sided Lipschitz
  bound (d-osL)}. We note that
\begin{itemize}
\item $f$ is osL with $b=0$ if and only if $f$ weakly decreasing;
\item if $f$ is Lipschitz with bound $\ell$, then $f$ is osL
  with $b=\ell$, whereas  the converse is false;
\item finally, for the scalar dynamics $ \dot{x} = f(x)$, the \GB lemma
  implies $|x(t)-y(t)|\leq \e^{bt} |x(0)-y(0)|$.
\end{itemize}
In what follows, we generalize this simple discussion in numerous
directions and study its implications.

\myclearpage\subsection{Contraction and Incremental Stability}

For a continuously differentiable $\map{f}{\real^n}{\real^n}$, consider
\begin{equation}
  \dot{x} = f(x).
\end{equation}
We next state the main theorem of contraction and exponential incremental stability.
\begin{theorem}[Equivalences on $(\real^n,\ell_2)$]
  \label{thm:contraction+IS-euclidean}
  For $P=P^\top\succ0$ and $c>0$, the following statements are equivalent:
\begin{enumerate}
\item\label{equiv:ell2:osL}  $(f(x) - f(y))^\top P (x - y)
  \leq -c \|x - y\|^2_{2,P^{1/2}}$, for all $x,y$;
  
\item\label{equiv:ell2:dosL}  $P \jac{f}(x) + \jac{f}(x)^\top
  P \preceq - 2c P$ for all $x$, or equivalently
  $\mu_{2,P^{1/2}}(\jac{f}(x)) \leq -c$ for all $x$;
  
\item\label{equiv:ell2:dIS} $D^+\|x(t) - y(t)\|_{{2,P^{1/2}}} \leq -c
  \|x(t) - y(t)\|_{{2,P^{1/2}}}$, for all solutions $x(\cdot), y(\cdot)$,
  where $D^+$ is the upper right Dini derivative;

\item\label{equiv:ell2:IS} $\|x(t) - y(t)\|_{{2,P^{1/2}}} \leq
  e^{-ct}\|x(0) - y(0)\|_{{2,P^{1/2}}}$, for all solutions $x(\cdot),
  y(\cdot)$.
\end{enumerate}
A vector field $f$ satisfying any and therefore all of these conditions is
said to be \emph{$c$-strongly contacting}.
\end{theorem}\smallskip

We refer to statement~\ref{equiv:ell2:osL} as the one-sided Lipschitz
condition (osL) and statement~\ref{equiv:ell2:dosL} as the differential
one-sided Lipschitz (d-osL) (a.k.a. the Demidovich condition). The last two
statements are about differential incremental stability (d-IS) and
exponential incremental stability (IS), respectively.

\begin{proof} We include an incomplete sketch of the proof.
  Statement~\ref{equiv:ell2:osL} implies \ref{equiv:ell2:dosL} by letting
  $y = x + hv$ for some $v \in \real^n$ and taking the limit as $h \to
  0^+$. Statement~\ref{equiv:ell2:dosL} implies \ref{equiv:ell2:dIS} by
  Coppel's inequality~\cite[Lemma~A]{CAD-HH:72}.  Statement~\ref{equiv:ell2:dIS} implies
  \ref{equiv:ell2:IS} by the \GB Comparison Lemma.
  Statement~\ref{equiv:ell2:IS} implies \ref{equiv:ell2:osL} by a Taylor
  expansion.
\end{proof}

Variations of Theorem~\ref{thm:contraction+IS-euclidean} hold for (1)
forward-invariant convex sets, (2) time-dependent vector fields, and (3)
non-differentiable vector fields $f$, where three of the four properties
remain equivalent: osL, d-IS, and IS.

For an affine $f(x)=Ax+b$, the osL condition reads
\begin{multline} \label{eq:contraction-affine-vf}
  (f(x) - f(y))^\top P (x-y) =   (x-y)^\top A^\top P (x-y) \\
  =  (x-y)^\top \frac{A^\top P+PA}{2} (x-y) \leq - c \norm{x-y}{2,P^{1/2}}^2. 
\end{multline}
Lumer's equalities~\eqref{eq:Lumer-equalities} imply that the smallest
number $-c$ ensuring the osL and d-osL conditions is
$-c=\mu_{2,P^{1/2}}(A)$.

\myclearpage\subsection{Consequences of Contraction: Equilibria} One of the
numerous desirable properties of strongly contracting vector fields is that
their flow forgets initial conditions (e.g., see
Figure~\ref{fig:contr-traj}) and, in the time-invariant case, globally
exponentially converges to a unique equilibrium point.  These points are
illustrated in the next result.

\begin{figure}[ht]\centering
  \includegraphics[width=.8\linewidth]{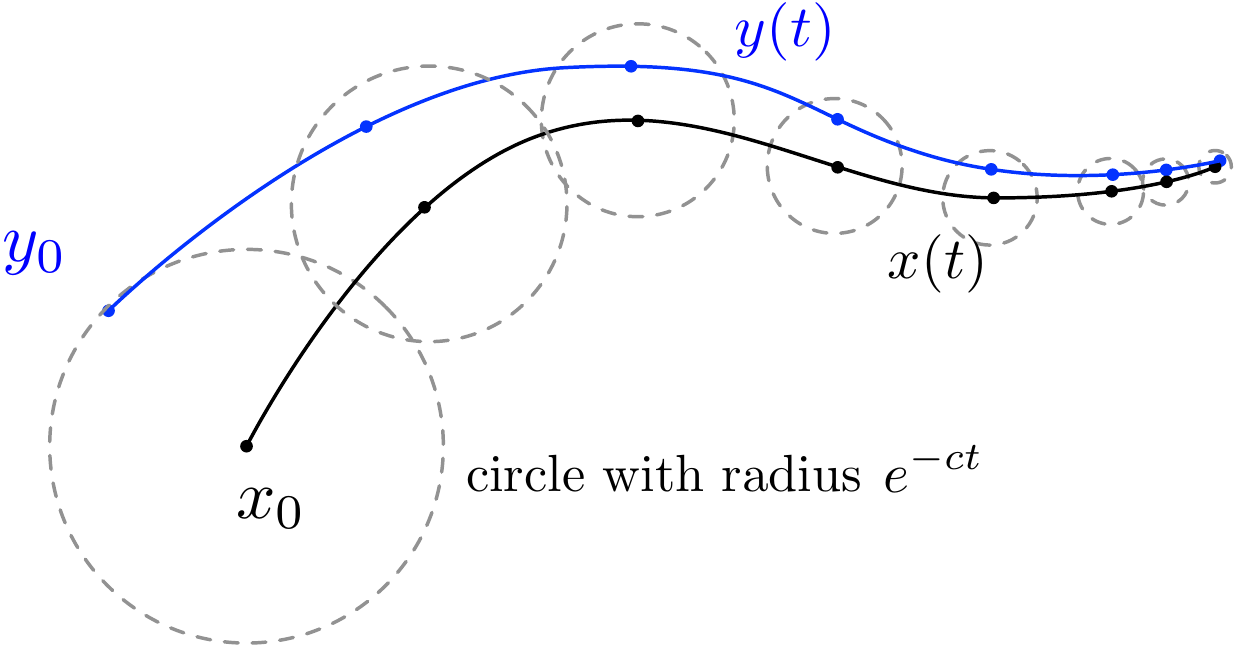}
  \caption{Exponential incremental stability of contracting vector
    fields. The distance between two trajectories decreases exponentially
    fast. }\label{fig:contr-traj}
\end{figure}

\begin{theorem}[Equilibria of contracting vector fields]\label{thm:equilibria:ell2}
  For a time-invariant vector field $f$ that is $c$-strongly contracting
  with respect to $\norm{\cdot}{2,P^{1/2}}$, $P=P^\top\succ0$,
  \begin{enumerate}
  \item\label{thm:eq:1} the flow of $f$ is a contraction, i.e., the
    distance between solutions exponentially decreases with rate $c$, and
  \item\label{thm:eq:2} there exists a unique equilibrium $x^*$ that is globally
    exponentially stable with global Lyapunov functions
    \begin{equation*} 
      x \mapsto  \norm{x-x^*}{2,P^{1/2}}^2
      \quad\text{ and }\quad
      x \mapsto  \norm{f(x)}{2,P^{1/2}}^2.
    \end{equation*} 
  \end{enumerate}
\end{theorem}\smallskip
\begin{proof}
  We include an incomplete sketch of the proof.
  Theorem~\ref{thm:contraction+IS-Riem}\ref{equiv:ell2:IS} immediately
  implies \ref{thm:eq:1} and that, for any positive $\tau$, the flow map of
  the vector field at time $\tau$ is a contraction map with constant
  $\e^{-c\tau}$. The fixed point of this contraction map is either a
  period orbit with period $\tau$ (which is impossible) or a fixed point of
  the flow map. The global Lyapunov functions follow from direct
  computation.
\end{proof}

\myclearpage\subsection{Equilibrium Computation via Forward Step Method}
The study of monotone operators is closely related to the study of
contracting vector fields.  As it is classic in the study of monotone
operators, we here aim to provide an algorithm to compute the equilibrium
points of a vector field $f$ (equivalently regarded as an operator):
\begin{equation}
  x^*\in \zero(f) \quad\iff\quad x^*\in\fixed(\Id + \alpha f),
\end{equation}
for any $\alpha>0$, where $\Id$ is the identity map.  Here we define
$\zero(f)=\setdef{x\in\real^n}{f(x)=0}$ and $\fixed(\Id+\alpha
f)=\setdef{x\in\real^n}{x=(\Id+\alpha{f})x}$.  A map $f$ is (globally)
\emph{$\ell$-Lipschitz continuous} if
\begin{equation}
  \norm{f(x) - f(y) }{2,P^{1/2}}\leq \ell \norm{x-y}{2,P^{1/2}}.
\end{equation}
for all $x,y$. We define the \emph{operator condition number} of a
$c$-strongly contracting and $\ell$-Lipschitz continuous map $f$ by
\begin{equation}
  \kappa =\ell/c\geq 1.
\end{equation}

\begin{remark}[Literature comparison]
\label{rem:monone-op}
In the literature on monotone operators, given $P=P^\top\succ0$, the map
$\map{g}{\real^n}{\real^n}$ is \myemph{$c$-strongly monotone} if
\begin{equation}
  (g(x) - g(y))^\top P (x-y) \geq c \norm{x-y}{2,P^{1/2}}^2.
\end{equation}
Clearly, $g$ is $c$-strongly monotone if and only if $-g$ is $c$-strongly
contracting.

  Next, we compare the operator condition number of a contracting affine
  $f(x)=Ax+b$, $A\in\real^{n\times{n}}$, with the standard contraction
  number of $A$.  First, recall that, given a norm $\norm{\cdot}{}$, the
  \emph{condition number} of a square invertible matrix $A$ is $\kappa(A)
  =\norm{A}{}\norm{A^{-1}}{}$.  Second, for the $P^{1/2}$-weighted $\ell_2$
  norm, we know from~\eqref{eq:contraction-affine-vf} that the contraction
  rate of $f$ equals $\mu_{2,P^{1/2}}(A)$.  Accordingly, given a norm
  $\norm{\cdot}{}$, the \emph{operator condition number} of a square matrix
  $A$ with $\mu(A)<0$ is
  \begin{equation}
    \kappa_\mu(A)= \frac{\norm{A}{}}{|\mu(A)|}.
  \end{equation}
  From \cite{CAD-HH:72}, note that $\mu(A)<0$ implies $\norm{A^{-1}}{}\leq
  1/|\mu(A)|$. Therefore, $\kappa(A) \leq \kappa_\mu(A)$.  One can show
  that the two condition numbers coincide for $A=A^\top$ and~$P=I_n$.
  \oprocend
\end{remark}

%% \paragraph{Matrix condition number}
%% Given a norm $\norm{\cdot}{}$, the \emph{condition number} of a square
%% invertible matrix $A$ is $\kappa(A) =\norm{A}{}\norm{A^{-1}}{}$.  For the
%% Euclidean norm $\ell_2$, $\kappa_2(A)=\frac{\smax(A)}{\smin(A)}$, where
%% $\smax$ and $\smin$ are the largest and smallest singular values of $A$.
%% If $A$ is symmetric, $\kappa_2(A)=\frac{|\lmax(A)|}{|\lmin(A)|}$, where
%% $\lmax$ and $\lmin$ are the largest and smallest eigenvalues of $A$ by
%% absolute value.

Given a start point $x_0\in\real^n$, the \emph{forward step method} for the
operator $f$, i.e., the explicit Euler integration algorithm for the vector
field $f$, is:
\begin{equation}
  \label{eq:forward}
  x_{k+1} = (\Id +\alpha f)x_k = x_k + \alpha f(x_k).
\end{equation}

\begin{theorem}{}\emph{(Optimal step size and contraction factor of forward step method)}
  \label{thm:forward-step-Euclidean}
  For $P=P^\top\succ0$, consider a map $\map{f}{\real^n}{\real^n}$ with
  strong contraction rate $c>0$, Lipschitz constant $\ell>0$, and condition
  number $\kappa =\ell/c$. Then
  \begin{enumerate}
  \item the map $\Id+\alpha f$ is a contraction map with respect to
    $\norm{\cdot}{2,P^{1/2}}$ for
    \begin{equation*}  0<\alpha<\frac{2}{c\kappa^2},
    \end{equation*}
  \item the step size minimizing the contraction factor and the minimum
    contraction factor (that is, the minimal Lipschitz constant of
    $\Id+\alpha f$) are
    \begin{equation}
    \begin{aligned}
     \subscr{\alpha}{E}^* &=  \frac{1}{c\kappa^2}, \\
      \subscr{\ell}{E}^*& = \Big(1-\frac{1}{\kappa^2}\Big)^{1/2}
      =1-\frac{1}{2\kappa^2} + \bigO\Big(\frac{1}{\kappa^4}\Big).
    \end{aligned}    
    \end{equation}    
  \end{enumerate}
\end{theorem}\medskip

\begin{proof}
  We only sketch the standard proof here:
  \begin{align*}
    \| (\Id  +\alpha & f)x -  (\Id+\alpha f)y \|_{2,P^{1/2}}^2
    \\
    &=
    \norm{ x-y + \alpha (f(x) - f(y)) }{{2,P^{1/2}}}^2
    \\
    &=    \norm{x-y}{{2,P^{1/2}}}^2
    + 2\alpha (f(x) - f(y))^\top P (x-y) \\
    &\qquad +\alpha^2\norm{f(x) - f(y)}{{2,P^{1/2}}}^2 
    \\
    &\leq (1-2\alpha c + \alpha^2 {\ell}^2) \norm{x-y}{{2,P^{1/2}}}^2.
  \end{align*}
  It is easy to check that $(1-2\alpha c + \alpha^2 \ell^2)<1$ if and only
  if $0<\alpha<2c/{\ell}^2$ and that the minimal contraction factor is
  $(1-c^2/{\ell}^2)^{1/2}$ at $\alpha^*=c/{\ell}^2$.
\end{proof}

\myclearpage
\section{Contraction Theory and Monotone Operators on Riemannian Manifolds}
\newcommand{\sfM}{\textsf{M}}
\renewcommand{\d}{\mathrm{d}}
\newcommand{\G}{\mathbb{G}}

In this section we consider a Riemannian manifold $(\sfM,\G)$ with
associated Levi-Civita connection $\nabla$, geodesic distance $\d_\G$, and
parallel transport $P(\gamma)$ along a geodesic arc $\gamma$.  Let
$\inprod{\cdot}{\cdot}_\G$ denote the inner product associated to $\G$ and
$\gamma'$ denote the velocity vector along a geodesic arc.
%% For a background on Riemannian geometry we refer to~\cite{MPdC:92}.

Loosely speaking, a vector field $X$ on a Riemannian manifold is
geodesically contracting ($-X$ is geodesically monotone) if the first
variation of the length of each geodesic arc $\gamma$, with infinitesimal
variation equal to the restriction of $X$ to $\gamma$, is nonpositive.

\begin{figure}[ht]\centering
  \includegraphics[width=.8\linewidth]{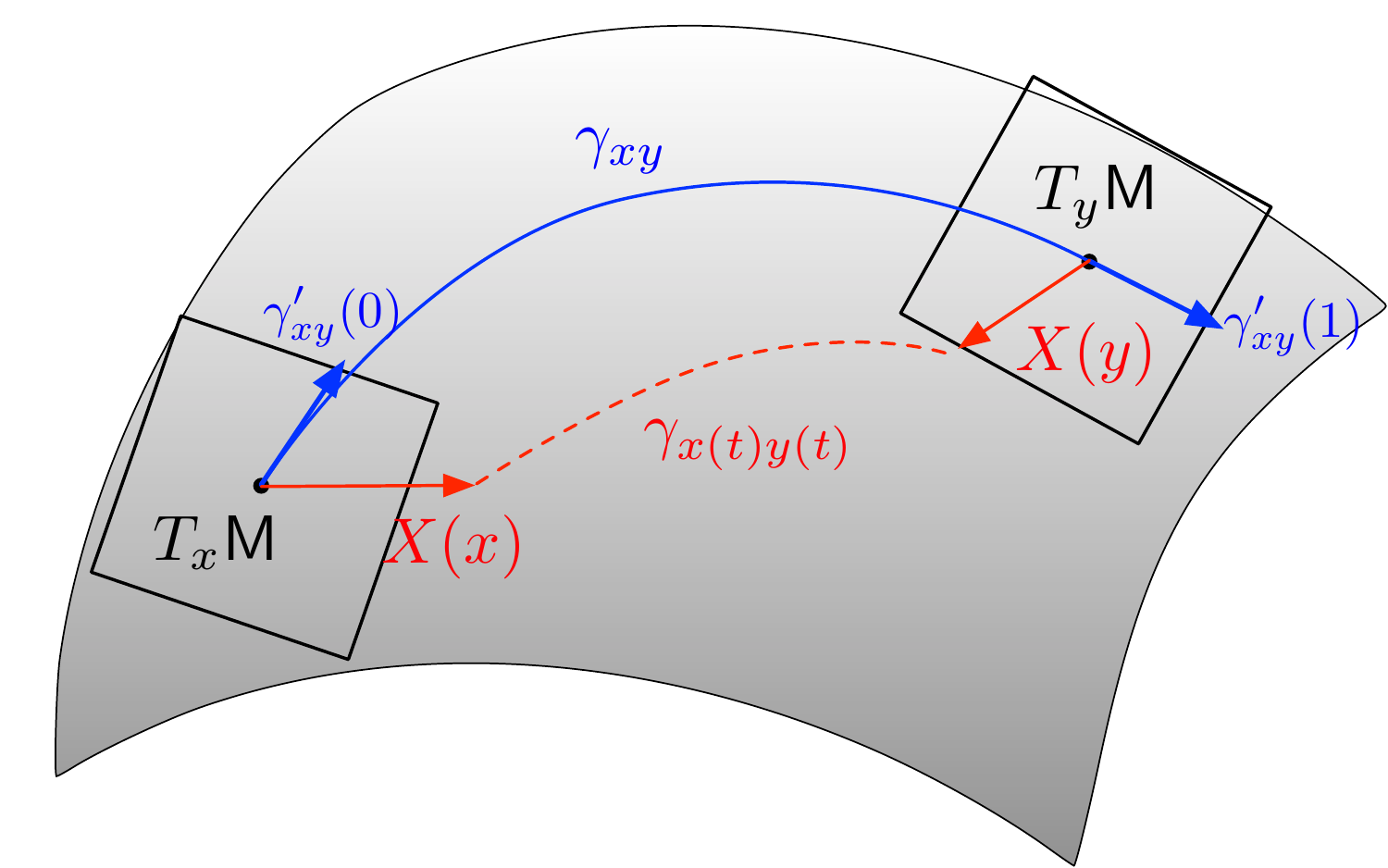}
  \caption{Contractivity of a vector field $X$ on a Riemannian manifold:
    the length of the geodesic curve $\gamma_{xy}$ connecting any two
    points $x$ and $y$ decreases along the flow of $X$, as a function of
    the inner product between $X$ and the geodesic velocity vector at $x$
    and $y$.}\label{fig:contr-traj}
\end{figure}

\myclearpage\subsection{Contraction and Incremental Stability}
We consider a time-independent vector field $X$ 
\begin{equation}
  \dot{x} = X(x).
\end{equation}

\begin{theorem}[Equivalences on $(\sfM,\G)$]
  \label{thm:contraction+IS-Riem}
  For a Riemannian manifold $(\sfM,\G)$ and $c>0$, the following statements
  are equivalent:
 \begin{enumerate}    
 \item\label{Riem:1}  for any $x,y\in\sfM$ and geodesic curve
   $\map{\gamma_{xy}}{[0,1]}{\sfM}$ with $\gamma_{xy}(0)=x$, $\gamma_{xy}(1)=y$,
   \begin{align*}
     &\inprod{X(y)}{\gamma_{xy}'(1)}_\G - \inprod{X(x)}{\gamma_{xy}'(0)}_\G
     \leq -c \, \d_\G(x,y)^2 ;
   \end{align*}
   
 \item\label{Riem:2}  for all $v_x\in T_x\sfM$
   \begin{equation*}
     \inprod{A_X(x)v_x}{v_x}_{\G} \leq -c  \norm{v_x}{\G}^2,
   \end{equation*}
   where $\map{A_X(x)}{T_x\sfM}{T_x\sfM}$ is the \emph{covariant
     differential of} $X$ defined by $A_X(x) v_x = \nabla_{v_x}X(x)$;
      
 \item\label{Riem:3} $D^+ \d_\G(x(t), y(t)) \leq -c\, \d_\G(x(t), y(t)) $,
   for all solutions $x(\cdot), y(\cdot)$;
   
 \item\label{Riem:4}  $\d_\G(x(t), y(t)) \leq e^{-ct} \d_\G(x(0), y(0)) $,
   for all solutions $x(\cdot), y(\cdot)$.
 \end{enumerate}
 A vector field $X$ satisfying any and therefore all of these conditions is
 said to be \emph{$c$-strongly contacting}.
\end{theorem}

\smallskip

\begin{proof}
  We refer to the appropriate references.  The equivalence between
  property~\ref{Riem:1} and property~\ref{Riem:2} is given in
  \cite{SZN:99b,JXDCN-OPF-LRLP:02}.  The implication
  \ref{Riem:2}$\implies$\ref{Riem:3} and \ref{Riem:4} is studied in
  \cite{JWSP-FB:12za}.  As before, the equivalence between
  statement~\ref{Riem:3} and statement~\ref{Riem:4} is independent of the
  vector field $X$ and related to the \GB comparison lemma.  
\end{proof}

%% Full details will be presented in future works.

Here are some comments drawing a parallel between
Theorems~\ref{thm:contraction+IS-Riem}
and~\ref{thm:contraction+IS-euclidean}. First, condition~\ref{Riem:1} is
known \cite{SZN:99b,JXDCN-OPF-LRLP:02} to be equivalent to either of the
following conditions
\begin{align*}
  & \inprod{\gamma_{xy}'(t)}{X(\gamma(t)}_\G +
  c \norm{\gamma'_{xy}(0)}{\G}^2 t \text{ is monotone decreasing},\\
  &   \inprod{ P(\gamma_{yx})_{y\to x} X(y) - X(x)}{\gamma_{yx}'(0)}_\G  
     \leq -c \, \d_\G(x,y)^2,
\end{align*}
where $\map{P(\gamma_{yx})_{y\to x}}{T_y\sfM}{T_x\sfM}$ is the parallel
transport along the geodesic from $y$ to $x$. It is easy to see that, when
$(\sfM,\G)$ is the Euclidean space with the standard $\ell_2$ inner
product, condition~\ref{Riem:1} coincides with the one-side Lipschitz
condition in
Theorem~\ref{thm:contraction+IS-euclidean}\ref{equiv:ell2:osL}.

Second, we clarify that statement~\ref{Riem:2} can easily be, and usually
is, written in components.  For every $x \in \sfM$ and in a coordinate
chart $(x^1,\ldots,x^n)$ in a neighborhood of $x$, statement~\ref{Riem:2}
is equivalent to the linear matrix inequality:
\begin{equation}\label{Eq:GeneralizedDemi}
  \left[\G_{ki}\frac{\partial X^k}{\partial x^\ell} + \frac{\partial
      X^k}{\partial x^i}\G_{k\ell} + \frac{\partial \G_{i\ell}}{\partial
      x^j}X^j\right] \preceq -2 c [\G_{i\ell}],
\end{equation}
or, in matrix form, letting $\G$ denote both the Riemannian metric as well
as its matrix coordinate representation,
\begin{equation}\label{Eq:GeneralizedDemi}
  \G(x) \jacfrac{X}{x}(x) + \jacfrac{X}{x}(x)^\top \G(x) + \dot{\G}(x)
  \preceq -2 c \G(x).
\end{equation}
This is the classic contraction condition given in~\cite{WL-JJES:98}, that
generalizes the classic Demidovich condition in
Theorem~\ref{thm:contraction+IS-euclidean}\ref{equiv:ell2:dosL}. The
parallel between Theorem~\ref{thm:contraction+IS-Riem}\ref{Riem:3}
and~\ref{Riem:4} versus
Theorem~\ref{thm:contraction+IS-euclidean}\ref{equiv:ell2:dIS}
and~\ref{equiv:ell2:IS} is evident.

\myclearpage\subsection{Consequences of Contraction: Equilibria}
In the interest of brevity we do not replicate Theorem~\ref{thm:equilibria:ell2},
whose extension to the Riemannian setting naturally holds.

%% In the interest of brevity we do not replicate
%% Theorem~\ref{thm:equilibria:ell2}, but we simply state that strongly
%% contractive vector fields on Riemannian manifolds enjoy the same properties
%% as those on $(\real^n,\ell_2)$, that is, the existence, uniqueness and
%% global exponential stability of an equilibrium point with two accessory
%% Lyapunov functions.

\myclearpage\subsection{Equilibrium Computation via Forward Step Method}
%% Take algorithm and theorem from~\cite{JHW-GL-VMM-CL:10}.

We start with two useful definitions.  Recall that a Riemannian manifold
$\sfM$ is complete if, for every $v_x\in T_x\sfM$, the geodesic curve
$\gamma_{v_x}(t)$ starting at $v_x$ at time $0$ is defined for all
$t\geq0$. Accordingly, the \emph{exponential map}
$\map{\exp_x}{T_x\sfM}{\sfM}$ is defined by $\exp_x(v_x)=\gamma_{v_x}(1)$.
A vector field $X$ is \emph{$\ell$-Lipschitz continuous} if
\begin{equation}
  \norm{ P(\gamma_{xy})_{x\to{y}} X(x)-X(y) }{\G} \leq  \ell \, \d_\G(x,y),
\end{equation}
for any $x,y\in\sfM$. Here we assume for simplicity that the geodesic
$\gamma_{xy}$ from $x$ to $y$ is unique.

Given a start point $x_0\in\sfM$, the \emph{forward step method} for the
operator $f$, i.e., the explicit Euler integration algorithm for the vector
field $f$, is:
\begin{equation}
  \label{eq:forward-sfM}
  x_{k+1} = \exp_{x_k}(\alpha X(x_k)).
\end{equation}

The following result is given in \cite[Theorem~5.1]{JXDCN-OPF-LRLP:02}.
\begin{theorem}{}\emph{(Riemannian forward step method)}
  \label{thm:forward-step}
  Consider a vector field $X$ on a Riemannian manifold $(\sfM,\G)$ with
  strong contraction rate $c>0$, Lipschitz constant $\ell>0$, and condition
  number $\kappa =\ell/c$. The sequence $\{x_k\}$ converges to the unique
  equilibrium point of $X$.
\end{theorem}\medskip

To the best of the authors' knowledge, it is an open conjecture whether the
algorithm $x\mapsto \exp_{x_k}(\alpha X(x_k))$ given in
equation~\eqref{eq:forward-sfM} is a Banach contraction mapping.

Practical implementations of the Riemannian forward step algorithms may
rely upon retractions (as an easily computable replacement of the
exponential map).

\myclearpage
\section{Contraction Theory and Monotone Operators on Non-Euclidean Spaces}
We now consider non-Euclidean spaces including, for example, $\real^n$
equipped with either the $\ell_1$ or $\ell_\infty$ norms.

\myclearpage\subsection{Linear Algebra Detour: Weak Pairings}
\label{subsec:WP}

\begin{table*}\centering\normalsize
  {\begin{tabular}{
	 p{0.2\textwidth}
	p{0.3\textwidth}
	 p{0.4\textwidth}
        %% ccc
      }
      Norm  & WP & Matrix measure and Lumer equality
      \\
      \hline
        \rowcolor{LightGray}
        $ \begin{aligned}
	\norm{x}{1} &= \sum_i |x_i| 
      \end{aligned}$
      &
      $\begin{aligned}
	\WP{x}{y}_{1} &= \|y\|_{1}\sign(y)^\top x
      \end{aligned}$
      & $\begin{aligned}
	\mu_{1}(A) &= \max_{j \in \until{n}} \Big(a_{jj} + \sum\nolimits_{i \neq j} |a_{ij}|\Big) \\
        &=  \sup_{\|x\|_{1} = 1}\sign(x)^\top Ax 
      \end{aligned}$
        \\
      $ \begin{aligned}
	\norm{x}{\infty} &= \max_i |x_i|
      \end{aligned}$
      &
      $\begin{aligned}
        \WP{x}{y}_{\infty} &= \max_{i \in I_{\infty}(y)} y_ix_i
      \end{aligned}$
      & $\begin{aligned}
	\mu_{\infty}(A) &=\max_{i \in \until{n}} \Big(a_{ii} + \sum\nolimits_{j \neq i} |a_{ij}|\Big) \\
        &=\max_{\|x\|_{\infty} = 1}			\max_{i \in I_{\infty}(x)} x_i(Ax)_i 
      \end{aligned}$
      \\\hline
  \end{tabular}}
  \caption{Table of norms, WPs, and matrix measures for $\ell_1$, and
    $\ell_\infty$ norms.  We define $\Iinfty(x) =
    \setdef{i\in\until{n}}{|x_i|=\norm{x}{\infty}}$.}
\end{table*}

We briefly review the notion and the properties of a weak pairing on
$\real^{n}$ from~\cite{AD-SJ-FB:20o}. A \emph{weak pairing (WP)} on
$\real^n$ is a map $\WP{\cdot}{\cdot}: \real^n \times \real^n \to \real$
satisfying:
\begin{enumerate}
\item\label{WP1}(Sub-additivity and continuity of first argument)
  $\WP{x_1+x_2}{y} \leq \WP{x_1}{y} + \WP{x_2}{y}$, for all
  $x_1,x_2,y \in \real^n$ and $\WP{\cdot}{\cdot}$ is continuous in its
  first argument,
\item\label{WP3}(Weak homogeneity)
  $\WP{\alpha x}{y} = \WP{x}{\alpha y} = \alpha\WP{x}{y}$ and
  $\WP{-x}{-y} = \WP{x}{y}$, for all
  $x,y \in \real^n, \alpha \geq 0$,
\item\label{WP4}(Positive definiteness) $\WP{x}{x} > 0$, for all
  $x \neq \vectorzeros[n],$
\item\label{WP5}(Cauchy-Schwarz inequality) \\
  $|\WP{x}{y}| \leq \WP{x}{x}^{1/2}\WP{y}{y}^{1/2}$, for all
  $x, y \in \real^n.$
\end{enumerate}
For every norm $\|\cdot\|$ on $\real^n$, there exists a (possibly not
unique) WP $\WP{\cdot}{\cdot}$ such that $\|x\|^2=\WP{x}{x}$, for every
$x\in \real^n$. When $\norm{\cdot}{}$ is the $\ell_2$ norm, the WP
coincides with the usual inner product. A WP $\WP{\cdot}{\cdot}$ satisfies
\emph{Deimling's inequality} if
\begin{equation*}
  \WP{x}{y} \le \|y\|\lim_{h\to 0^{+}}\frac{\|y+hx\|-\|y\|}{h},
\end{equation*}
for every $x,y\in \real^n$. A WP satisfying Deimling's inequality also
satisfies, for all $A \in \real^{n\times n}$,  the Lumer's equality
\begin{equation}
\mu(A) = \sup_{x\neq \vectorzeros[n]} \frac{\WP{Ax}{x}}{\norm{x}{}^2}.
\end{equation}
For invertible $R\in \real^{n\times n}$, we define the \emph{weighted sign
  WP} $\WP{\cdot}{\cdot}_{1,R}$ and the \emph{weighted max WP}
$\WP{\cdot}{\cdot}_{\infty,R}$ by
\begin{align}\label{eq:1-R}
\WP{x}{y}_{1,R}&= \|Ry\|_1\mathrm{sign}{(Ry)}^{\top}Rx,\\
\WP{x}{y}_{\infty,R}&= \max_{i\in I_{\infty}(Ry)} (Rx)_i(Ry)_i,
\end{align}
where $I_{\infty}(x)=\setdef{i\in \{1,\ldots,n\}}{x_i=\|x\|_{\infty}}$. 
It can be shown that, for $p\in
\{1,\infty\}$ and invertible matrix $R\in \real^{n\times n}$, we have
$\|Rx\|^2_{p}=\WP{x}{x}_{p,R}$ and $\WP{\cdot}{\cdot}_{p,R}$
satisfies Deimling's inequality. We refer to~\cite{AD-SJ-FB:20o} for a
detailed discussion on WPs and formulas for arbitrary $p \in [1,\infty]$.

\myclearpage\subsection{Contraction and Incremental Stability}
For a continuously differentiable $\map{f}{\real^n}{\real^n}$, consider
\begin{equation}
  \dot{x} = f(x).
\end{equation}

\begin{theorem}[Equivalences on $(\real^n,\norm{\cdot}{})$]
  \label{thm:contraction+non-euclidean}
  For a norm $\norm{\cdot}{}$ with matrix measure $\mu(\cdot)$ and
  compatible WP $\WP{\cdot}{\cdot}$ satisfying Deimling's inequality,
  and $c>0$, the following statements are equivalent:
  \begin{enumerate}
  \item\label{ctGen:5} $\WP{f(x) - f(y)}{x - y} \leq -c\|x - y\|^2$
    for all $x,y$,
  \item\label{ctGen:4} $\WP{\jac{f}(x)v}{v} \leq -c\|v\|^2$, for all $v,
    x$, or\newline $\mu(\jac{f}(x)) \leq -c$, for all $x$,
  \item\label{ctGen:6} $D^+\|x(t) - y(t)\| \leq -c\|x(t) - y(t)\|$, for all
    solutions $x(\cdot), y(\cdot)$,
  \item\label{ctGen:1} $\|x(t) - y(t)\| \leq e^{-c(t)}\|x(0) - y(0)\|$, for
    all solutions $x(\cdot), y(\cdot)$ .
  \end{enumerate}
\end{theorem}

\begin{proof}
  We refer the reader to~\cite{AD-SJ-FB:20o}.
\end{proof}

%% $\map{f}{\real^n}{\real^n}$ is a continuously differentiable vector field
%% with Jacobian $\jac{f}$. Each row contains three equivalent statements, to
%% be understood for all $x,y\in\real^n$ and all $v\in\real^n$. The function
%% $\map{\sign}{\real^n}{\{-1,0,1\}^n}$ is the entrywise $\sign$ function,
%% $\circ$ is the entrywise product, the absolute value and power of a vector
%% are the entrywise absolute value and power, respectively.

\myclearpage\subsection{Consequences of Contraction: Equilibria} In the
interest of brevity we do not replicate Theorem~\ref{thm:equilibria:ell2},
whose extension to non-Euclidean setting naturally holds.

\myclearpage\subsection{Equilibrium Computation via Forward Step Method}

Consider the continuously differentiable dynamics $\dot{x} = f(x)$. Let
$\norm{\cdot}{}$ denote a norm with compatible WP $\WP{\cdot}{\cdot}$.
Assume the vector field $f$ is $c$-strongly contracting, i.e.,
\begin{equation}
  \label{eq:strong-contraction}
  \WP{f(x)-f(y)}{x-y} \leq -c \norm{x-y}{}^2,
\end{equation}
and (globally) Lipschitz continuous with constant $\ell$, i.e.,
\begin{equation}
    \label{eq:strong-smoothness}
  \norm{f(x)-f(y)}{} \leq \ell \norm{x-y}{},
\end{equation}
for any $x,y$. Next we summarize Theorem~1 from
\cite{SJ-AD-AVP-FB:21f}.

\begin{table*}\centering\normalsize
  {\begin{tabular}{
	p{0.2\linewidth}
	p{0.3\linewidth}
	p{0.4\linewidth}
      }
      Measure & Demidovich & One-sided Lipschitz \\
      bound &  condition &  condition \\
      \hline
      \rowcolor{LightGray}    & &  \\[-1ex]
      \rowcolor{LightGray}
      $ \ds \mu_{2,P^{1/2}}(\jac{f}(x))\leq b$ &
      $\ds P \jac{f}(x) + \jac{f}(x)^\top  P \preceq 2 b P $
      & $\ds (x-y)^\top  P \big( f(x) - f(y) \big) \leq b \norm{x-y}{P^{1/2}}^2$
      \\[2ex]
      \rowcolor{White} && \\[-1ex]
      \rowcolor{White}
      %% $ \ds \mu_{p}(\jac{f}(x))\leq b$
      %% & $\ds (v \circ |v|^{p-2})^\top \jac{f}(x)v \leq b\|v\|_{p}^p $
      %% & $\ds ((x - y) \circ |x-y|^{p-2})^\top (f(x) - f(y)) \leq b\|x -
      %% y\|_{p}^p$
      %% \\[2ex]
      %% \rowcolor{LightGray} && \\[-1ex]
      %% \rowcolor{LightGray}
      $\ds \mu_{1}(\jac{f}(x))\leq b$
      &  $\ds \sign(v)^{\top} \jac{f}(x) v\le b \norm{v}{1}$
      &  $\ds \sign(x-y)^{\top} (f(x) - f(y))\le b \norm{x-y}{1}$
      \\[2ex]
      \rowcolor{LightGray} && \\[-1ex]
      \rowcolor{LightGray}
      $\ds \mu_{\infty}(\jac{f}(x))\leq b$
      &  $\ds     \max_{i\in \Iinfty(v)}\! v_i \left(\jac{f}(x) v\right)_i
      \le b \norm{v}{\infty}^2$ 
      &
      $\ds\max_{i\in\Iinfty(x-y)} \! (x_i-y_i) (f_i(x)-f_i(y)) \leq b \norm{x-y}{\infty}^2$
      \\   \hline
      % \\[1ex]       \rowcolor{LightGray} && 
  \end{tabular}}
  \caption{Table of equivalences between measure bounded Jacobians,
    differential Demidovich and one-sided Lipschitz conditions.}
\end{table*}

\begin{theorem}[Forward step method on WP spaces]\label{thm:forwardstep-WP}
  Consider a norm $\norm{\cdot}{}$ with compatible WP $\WP{\cdot}{\cdot}$.
  Let the continuously differentiable function $f$ be $c$-strongly
  contracting, have Lipschitz constant $\ell$, and have condition number
  $\kappa = \ell/c\geq 1$. Then
  \begin{enumerate}
\item the map $\Id+\alpha f$ is a contraction map with respect to
  $\norm{\cdot}{}$ for
  \begin{equation*}  0<\alpha<\frac{1}{c\kappa(1+\kappa)},
  \end{equation*}
\item the step size minimizing the contraction factor and the minimum
  contraction factor are
  \begin{align}
    \subscr{\alpha}{nE}^* &= \frac{1}{c}\Big(\frac{1}{2\kappa^2} -
    \frac{3}{8\kappa^3} +  \bigO\Big(\frac{1}{\kappa^4}\Big) \Big), \label{eq:alpha-nE} \\
    \subscr{\ell}{nE}^* &= 1 -
    \frac{1}{4\kappa^2} + \frac{1}{8\kappa^3} + \bigO\Big(\frac{1}{\kappa^4}\Big).
  \end{align}    
  \end{enumerate}
\end{theorem}\medskip

Compared to the forward step method for contracting systems in the
Euclidean space in Theorem~\ref{thm:forward-step-Euclidean}, the optimal
step size is smaller (by a factor of 2 and by higher order terms) and the
optimal contraction factor is larger (the gap is larger by a factor of $2$
and by higher order terms).

\begin{example}
  Consider the affine system $\dot{x}=Ax+b$, where $A=\begin{bmatrix} -10
  &2.5 \\ 9 & -3 \end{bmatrix}$ and $b=\begin{bmatrix}-19
  \\ 20 \end{bmatrix}$. We compute
  \begin{align*}
    \mu_2(A) = \lambda_{\max}(\tfrac{1}{2}(A+A^{\top})) = \lambda_{\max}\begin{bmatrix}
      -10 &5.75 \\ 5.75 & -3 
    \end{bmatrix} = 0.231.
  \end{align*}
  Therefore, this system is not contracting with respect to $\ell_2$ norm
  and Theorem~\ref{thm:forward-step-Euclidean} is not applicable for
  finding its equilibrium point. However,
  \begin{align*}
    \mu_1(A) = -0.5 < 0. 
  \end{align*}
  Moreover, we have $\|A(x-y)\|_1\le \|A\|_1\|x-y\|_1$. Thus, with respect
  to the $\ell_1$ norm, the affine system is strongly contracting with rate
  $0.5$ and Lipschitz continuous with Lipschitz constant $\|A\|_1$. Now we
  can use Theorem~\ref{thm:forwardstep-WP} for the $\ell_1$ norm and show
  that $I_2 + \alpha (Ax+b)$ is contracting for every $\ds 0<\alpha <
  \frac{|\mu_1(A)|}{\|A\|_1(|\mu_1(A)| + \|A\|_1)}$. \oprocend
\end{example}

\smallskip

It is an open conjecture whether a version of
Theorem~\ref{thm:forwardstep-WP} holds for nonsmooth vector fields. We
refer to~\cite{SJ-AD-AVP-FB:21f} for additional results on the
optimal step size and acceleration results for the norms $\ell_1$ and
$\ell_\infty$.

\myclearpage\subsection{Comments on Implicit Algorithms} We here review the
implicit Euler integration scheme and show its basic properties for
strongly contracting vector fields; the original reference for this
material is~\cite{CAD-HH:72}. Given a vector field $f$ on $\real^n$, we
(implicitly) define the sequence:
\begin{equation}
  \label{eq:Euler-implicit}
  x_{k+1} = x_k + \alpha f(x_{k+1}).
\end{equation}
This scheme corresponds to the operator $(\Id - \alpha f)^{-1}$.
\begin{theorem}[Implicit Euler method on WP spaces]
  Let $\|\cdot\|$ denote a norm with compatible WP $\WP{\cdot}{\cdot}$. Let
  $f$ be a $c$-strongly contracting vector field with unique equilibrium
  point $x^*$ and Lipschitz constant $\ell$. Then
  \begin{enumerate}
  \item\label{implE:cute} the $(\Id - \alpha f)^{-1}$ is a contraction
    mapping with contraction factor $(1+\alpha c)^{-1}$ for any $\alpha>0$;
  \item\label{implE:fixedp} if $\alpha\ell<1$, then, at each time $k$, the
    implicit equation~\eqref{eq:Euler-implicit} is well-posed and the
    fixed-point iteration $x_{k+1}^{[0]}=x_k$, $x_{k+1}^{[i+1]} = x_k +
    \alpha f(x_{k+1}^{[i]})$ is a contraction mapping with contraction
    factor $\alpha\ell$;
  \item\label{implE:Newton} if $\alpha\ell<1$ and
    $\norm{f(x_0)}{}\leq\tfrac{2(1+\alpha
    c)(1-\alpha\ell)}{\alpha(1+\alpha\ell)}$, then, at each time $k$, the
    Newton-Raphson iteration $x_{k+1}^{[0]}=x_k$, $x_{k+1}^{[i+1]} =
    x_{k+1}^{[i]} - \jac{g}(x_{k+1}^{[i]})^{-1}(g(x_{k+1}^{[i]})-x_k)$, for
    $g(x)=x-\alpha f(x)$, converges quadratically to the solution the
    implicit equation~\eqref{eq:Euler-implicit}.
  \end{enumerate}
\end{theorem}\medskip
\begin{arxiv}
  \begin{proof}
    Given two sequences $\{x_k\}_{k=1}^{\infty}$ and $\{y_k\}_{k=1}^\infty$ generated
    by~\eqref{eq:Euler-implicit}, the properties of WPs in~\ref{subsec:WP}
    imply:
    \begin{align*}
      &\norm{x_{k+1}-y_{k+1}}{}^2 \\
      &\qquad = \WP{x_{k}-y_{k}+\alpha(f(x_{k+1})-f(y_{k+1}))}{x_{k+1}-y_{k+1}} \\
      &\qquad \leq
      \WP{x_{k}-y_{k}}{x_{k+1}-y_{k+1}} \\
      & \qquad\qquad +\alpha \WP{f(x_{k+1})-f(y_{k+1})}{x_{k+1}-y_{k+1}} \\
      &\qquad \leq
      \norm{x_{k}-y_{k}}{}\norm{x_{k+1}-y_{k+1}}{} -c \alpha \norm{x_{k+1}-y_{k+1}}{}^2.
    \end{align*}
    After simple manipulation we obtain
    $\norm{x_{k+1}-y_{k+1}}{}\leq(1+c\alpha)^{-1}\norm{x_{k}-y_{k}}{}$;
    this proves~\ref{implE:cute}; for a more general treatment
    see~\cite{PCV-FB:21d-simple}.  The proof of
    statement~\ref{implE:fixedp} is immediate, since the Lipschitz constant
    of $x\mapsto x+\alpha f(x)$ is $\alpha\ell$.  The proof of
    statement~\ref{implE:Newton} relies upon~\cite[Theorem~C]{CAD-HH:72}
    and is omitted in the interest of brevity.
  \end{proof}
   
  A conjecture is that the Newton-Raphson iteration converges globally and
  not only locally.
   
\end{arxiv}

\myclearpage\subsection{Comments on Higher Order Algorithms}
%% (see \cite{KH-SZ:21}):

We here briefly present the extra-gradient algorithm and prove that it has
accelerated convergence over the forward step method. Let $f$ be a vector
field on $\real^n$. The \emph{extra-gradient iterations} with step size
$\alpha$ are given by
\begin{equation}
  \label{eq:extra-gradient}
  \begin{aligned}
    x_{k+0.5} &= x_k + \alpha f(x_k), \\
    x_{k+1} &= x_k + \alpha f(x_{k+0.5}).
  \end{aligned}
\end{equation}
\begin{theorem}[Extra-gradient method on WP spaces]
  Let $\|\cdot\|$ denote a norm with compatible WP $\WP{\cdot}{\cdot}$. Let
  $f$ be a $c$-strongly contracting vector field with unique equilibrium
  point $x^*$, Lipschitz constant $\ell$, and condition number
  $\kappa=\tfrac{\ell}{c}\ge 1$. Then
  \begin{enumerate}
  \item\label{p1:extragradient-convergence} the extra-gradient
    iterations~\eqref{eq:extra-gradient} satisfy
    \begin{align*}
      \|x_{k+1}-x^*\| \le \frac{1+\alpha^3\ell^3}{1+\alpha c}\|x_k-x^*\|
    \end{align*}
    and, for every $0 \le \alpha \le \frac{1}{c\kappa\sqrt{\kappa}}$, the
    sequence $\{x_k\}_{k=0}^{\infty}$ converges to $x^*$;
  \item\label{p2:extragradient-rate} for $\alpha = \frac{1}{2
    c\kappa\sqrt{\kappa}}$, the convergence factor is
    \begin{align*}
      1 - \frac{3}{8\kappa\sqrt{\kappa}} + \bigO\Big(\frac{1}{\kappa^3}\Big).
    \end{align*}
  \end{enumerate}
\end{theorem}\medskip

The proof of this theorem is omitted in the interest of brevity.  It is an
open conjecture whether the optimal convergence factor is of order
$1-1/\kappa$.

\myclearpage
\section{Conclusions}

Contraction theory and monotone operator theory are well established
methodologies to tackle control, optimization and learning problems. This
article surveys connections among them and shows how to generalize some
elements of these theories to Riemannian manifolds and non-Euclidean norms.

% open conjecture: some of these algorithms to compute periodic orbits,
% when system is subject to periodic forcing

{\renewcommand{\baselinestretch}{0.92}
\bibliographystyle{plainurl+isbn}
\bibliography{alias,Main,FB}}

\end{document}